\documentclass[sn-mathphys]{sn-jnl}% Math and Physical Sciences Reference Style
\usepackage{hyperref}

\jyear{2023}%

\theoremstyle{thmstyleone}%
\newtheorem{theorem}{Theorem}%  meant for continuous numbers
\theoremstyle{thmstyletwo}%
\newtheorem{example}{Example}%
\newtheorem{counterexample}{Counterexample}%
\newtheorem{remark}{Remark}%
\newtheorem{question}{Question}
\theoremstyle{thmstylethree}%
\newtheorem{definition}{Definition}%
\newtheorem{lemma}{Lemma}%
\newtheorem{corollary}{Corollary}%

\begin{document}

\title[Value distribution of certain differential polynomials...]{Value distribution of certain differential polynomials leading to some normality criteria}

%%=============================================================%%
%% Prefix	-> \pfx{Dr}
%% GivenName	-> \fnm{Joergen W.}
%% Particle	-> \spfx{van der} -> surname prefix
%% FamilyName	-> \sur{Ploeg}
%% Suffix	-> \sfx{IV}
%% NatureName	-> \tanm{Poet Laureate} -> Title after name
%% Degrees	-> \dgr{MSc, PhD}
%% \author*[1,2]{\pfx{Dr} \fnm{Joergen W.} \spfx{van der} \sur{Ploeg} \sfx{IV} \tanm{Poet Laureate} 
%%                 \dgr{MSc, PhD}}\email{iauthor@gmail.com}
%%=============================================================%%

\author[1]{\fnm{Kuldeep Singh} \sur{Charak}}\email{kuldeepcharak65@gmail.com\\
Orcid: 0000-0003-0151-3249}

\author*[1]{\fnm{Nikhil} \sur{Bharti}}\email{nikhilbharti94@gmail.com\\
Orcid: 0000-0003-2501-6247}
%\equalcont{These authors contributed equally to this work.}

\affil[1]{\orgdiv{Department of Mathematics}, \orgname{University of Jammu}, \orgaddress{\city{Jammu}, \postcode{180006}, \state{Jammu and Kashmir}, \country{India}}}

%\affil[3]{\orgdiv{Department}, \orgname{Organization}, \orgaddress{\street{Street}, \city{City}, \postcode{610101}, \state{State}, \country{Country}}}

%%==================================%%
%% sample for unstructured abstract %%
%%==================================%%

\abstract{In this paper, we prove some normality criteria concerning transitivity of normality from one family of meromorphic functions to another which improve and generalize some recent results. We also prove some value distribution results for certain differential polynomials which lead to some normality criteria involving sharing of holomorphic functions with certain differential polynomials. As a consequence, a counterexample to the converse of the Bloch's principle is also given.}


%%\pacs[JEL Classification]{D8, H51}


\maketitle
\section{Introduction and Statement of Results}
For the sake of convenience, we shall denote by  $\mathcal{H}(D),$ the class of all holomorpic functions on the domain $D$ in $\mathbb{C}$ and by $\mathcal{M}(D),$ the class of all meromorphic functions on the domain $D.$ $\mathbb{D}$  and $D(a,r)$ shall denote the open unit disk and the open disk  with center $a$ and radius $r,$ in $\mathbb{C}$, respectively. We assume that the reader is familiar with the standard notations of the Nevanlinna value distribution theory of meromorphic functions, like $m(r,f),~N(r,f),~T(r,f)$ (see \cite{hayman, lo}). 

\medskip
 
A family $\mathcal{F}\subset\mathcal{M}(D)$ is said to be {\it normal} in $D$ if every sequence of functions in $\mathcal{F}$ has a subsequence which converges locally uniformly in $D$ with respect to the spherical metric. The limit function is either meromorphic in $D$ or identically $\infty.$ If $\mathcal{F}$ happens to be a family of holomorphic functions, then the Euclidean metric can be taken in place of the spherical metric, and in this case, the limit function is either holomorphic in $D$ or identically $\infty$ (see \cite{schiff, zalcman-1, zalcman-2}). 

\smallskip

Let $k$ be a positive integer and let $n_0, n_1,\ldots, n_k$ be non-negative integers,  not all zero. Let $f\in\mathcal{M}(D).$ Then the expression of the form $$M[f]:=a\cdot\prod\limits_{j=0}^{k}\left(f^{(j)}\right)^{n_j}$$ is called a differential monomial of $f,$ where $a~(\not\equiv 0, \infty)\in\mathcal{M}(D).$ If $a\equiv 1,$ then $M[f]$ is said to be a normalized differential monomial of $f.$ The quantities $$\gamma_M:=\sum\limits_{j=0}^{k}n_j \mbox{ and } \Gamma_M:=\sum\limits_{j=0}^{k}(j+1)n_j$$ are called the degree and weight of the monomial $M[f],$ respectively. 

For $1\leq i\leq m,$ let $M_i[f]=\prod\limits_{j=0}^{k}\left(f^{(j)}\right)^{n_{ji}}$  be $m$ differential monomials of $f.$ Then the sum $P[f]:= \sum\limits_{i=1}^{m}a_iM_i[f]$ is called a differential polynomial of $f$ and the quantities $\gamma_P:=\max\left\{\gamma_{M_i}: 1\leq i\leq m\right\},$ $\nu_P:=\min\left\{\gamma_{M_i}: 1\leq i\leq m\right\}$ and $\Gamma_P:=\max\left\{\Gamma_{M_i}:1\leq i\leq m\right\}$ are the degree, lower degree and weight of the differential polynomial $P[f],$ respectively. The number $k,$ which is the highest order of derivative occurring in the differential polynomial $P[f]$ shall be called  differential order of $P[f].$ Further, we shall denote by $\Theta_P,$ the ratio of the weight of the differential polynomial to its lower degree. That is, 
$$\Theta_P=\frac{\Gamma_P}{\nu_P}=\max\left\{\frac{\Gamma_{M_i}}{\gamma_{M_i}}:1\leq i\leq m\right\}.$$

\medskip

In the present paper, where not otherwise stated, we consider the differential polynomials of the form 
\begin{equation}\label{eqn:1}
P[f]=\sum\limits_{i=1}^{m}a_i~M_i[f],
\end{equation}
where $$M_i[f]=\prod\limits_{j=0}^{k}\left(f^{(j)}\right)^{n_{ji}},~1\leq i\leq m~$$ are normalized differential monomials satisfying 
\begin{equation}\label{eqn:2}
\frac{\Gamma_{M_t}}{\gamma_{M_t}}<\frac{\Gamma_{M_1}}{\gamma_{M_1}}=\Theta_P ~\mbox{ for }~ 2\leq t\leq m,
\end{equation}
and the coefficients $a_{i} \in \mathcal{H}(D)$  with $a_1(z)\neq 0.$
\begin{definition}
Let $f~g\in\mathcal{M}(D).$ Then $f$ and $g$ are said to share the function $h$ in $D$ if $Z(h,f)=Z(h,g),$ where $Z(h, \psi):=\left\{\zeta\in D: \psi(\zeta)-h(\zeta)=0\right\}$ is the set of zeros of $\psi-h$ in $D$ counted with ignoring multiplicities; if the zeros are counted with their due multiplicities, then we say that  $f$ and $g$ share $h$ with counting multiplicities and we write this as $f$ and $g$ share $h$ CM. If $Z(h,f)\subseteq Z(h,g),$ then we say that $f$ and $g$ partially share $h$ in $D.$
\end{definition}

In $2013,$ Liu et al. \cite{liu} initiated the study on transitivity of normality from one family of meromorphic functions to another under suitable conditions and obtained the following results:
\begin{theorem}\label{thm:liu-1}
    Let $\mathcal{F},~\mathcal{G}\subset\mathcal{M}(D)$ and $a_i~(i=1,2,3,4)$ be four distinct complex numbers. Assume that $\mathcal{G}$ is normal in $D.$ If for every $f\in\mathcal{F},$ there exists $g\in\mathcal{G},$ such that $f$ and $g$ share $a_i~(i=1,2,3,4)$  in $D$ , then $\mathcal{F}$ is normal in $D.$
\end{theorem}

\begin{theorem}\label{thm:liu-2}
Let $\mathcal{F},~\mathcal{G}\subset\mathcal{M}(D)$ be two families, all of whose zeros have multiplicities at least $k+1,~k\in\mathbb{N}.$ Let $a$ be a non-zero complex number. Assume that the family $\mathcal{G}$ is normal in $D$ such that no subsequence of $\mathcal{G}$ converges spherically locally uniformly to $\infty$ or to a function $g$ satisfying $g^{(k)}\equiv a.$  If for every $f\in\mathcal{F},$ there exists $g\in\mathcal{G}$ such that $f$ and $g$ share $0$ and $\infty,$ and $f^{(k)}$ and $g^{(k)}$ share $a$ CM, then $\mathcal{F}$ is normal in $D.$  
\end{theorem}

Theorem \ref{thm:liu-1} has been improved where the four distinct shared values are replaced by four distinct shared holomorphic functions (see \cite[Theorem 1]{chen-xu}) and also by four uniformly separated proximate 
values which may depend on each $f\in\mathcal{F}$ (see \cite[Theorem 7]{kumar}). On the other hand, by a beautiful application of complex dynamics, Chang \cite[Theorem 1.4]{chang-2} established that for $k=1,$ the condition ``$f^{(k)}$ and $g^{(k)}$ share $a$ CM" in Theorem \ref{thm:liu-2} can be replaced by the condition ``$f^{(k)}$ and $g^{(k)}$ share $a$". Following the ideas of Chang \cite{chang-2}, Chen and Xu \cite[Theorem 2]{chen-xu} improved Theorem \ref{thm:liu-2} as:

\begin{theorem}\label{thm:chen-xu}
    Let $\mathcal{F},~\mathcal{G}\subset\mathcal{M}(D)$ be two families, all of whose zeros have multiplicities at least $k+1,~k\in\mathbb{N}.$ Let $a$ be a non-zero complex number. Assume that the family $\mathcal{G}$ is normal in $D$ such that no subsequence of $\mathcal{G}$ converges spherically locally uniformly to $\infty$ or to a function $g$ satisfying $g^{(k)}\equiv a,$ and poles of $f\in\mathcal{F}$ have multiplicities at least $k.$  If for every $f\in\mathcal{F},$ there exists $g\in\mathcal{G}$ such that $f$ and $g$ partially share $0,$ $f$ and $g$ share $\infty,$ and $f^{(k)}$ and $g^{(k)}$  partially share $a,$ then $\mathcal{F}$ is normal in $ D.$
\end{theorem}
Recently, Ahamed and Mandal \cite[Theorem 2.2]{ahamed} considered families of holomorphic functions and generalized Theorem \ref{thm:liu-2} to differential monomials as follows:
 
\begin{theorem}\label{thm:mola}
Let $\mathcal{F},~\mathcal{G}\subset\mathcal{H}(D)$ be two families, all of whose zeros have multiplicities at least $k+1,~k\in\mathbb{N}.$ Let $a$ be a non-zero complex number. Assume that the family $\mathcal{G}$ is normal in $D$ such that for any subsequence $g_n$ of $\mathcal{G},~g_n\rightarrow g,~g\not\equiv\infty$ and $M[g]\not\equiv a.$ If for every $f\in\mathcal{F},$ there exists $g\in\mathcal{G}$ such that $f$ and $g$ share $0,$ and $M[f]$ and $M[g]$ share $a,$ then $\mathcal{F}$ is normal in $D.$  
\end{theorem}

About Theorem \ref{thm:mola}, it is natural to ask the following question:

\begin{question}\label{q:1}
Is the family $\mathcal{F}$ normal in $D$ if the differential monomials $M[f]$ and $M[g]$ sharing a non-zero complex number $a$ are replaced by some differential polynomials partially  sharing $a$ in $D$? 
\end{question}

We answer Question \ref{q:1} as follows:

\begin{theorem}\label{thm:2}
Let $\mathcal{F},~\mathcal{G}\subset\mathcal{H}(D)$ be two families, all of whose zeros have multiplicities at least $k+1,~k\in\mathbb{N}.$ Let $P[f]$ be a differential polynomial defined in \eqref{eqn:1} and satisfying \eqref{eqn:2}. Let $a$ be a non-zero complex number. Assume that the family $\mathcal{G}$ is normal in $D$ such that for any subsequence $\left\{g_n\right\}\subset\mathcal{G},~g_n\rightarrow g,~g\not\equiv\infty$ and $P[g]\not\equiv a.$ If for every $f\in\mathcal{F},$ there exists $g\in\mathcal{G}$ such that $f$ and $g$ partially share $0$, and,  $P[f]$ and $P[g]$ partially share $a$ in $D,$ then $\mathcal{F}$ is normal in $D.$
\end{theorem}

The following examples show that various conditions in the hypothesis of Theorem \ref{thm:2} are essential.

\begin{example}\label{exp:1}
Let $k$ be a positive integer and $a=1.$ Let $$\mathcal{F}=\left\{f_n:f_n(z)=\frac{nz^{k+1}}{(k+1)!},~n\in\mathbb{N}\right\}$$ and 
$$\mathcal{G}=\left\{g_n:g_n(z)=\frac{z^{k+1}}{2n(k+1)!},~n\in\mathbb{N}\right\}$$
be two families of holomorphic functions in $\mathbb{D}.$ Then for each $n,$ $f_n$ and $g_n$ have zeros of multiplicities at least $k+1$ and $g_n\rightarrow g\equiv 0.$ Clearly, the family $\mathcal{G}$ is normal in $\mathbb{D}.$ Let $P[f]:=(k+1)!ff^{(k)}.$ Then $$P[f_n](z)=(k+1)!f_nf_n^{(k)}(z)=n^2z^{k+2} \mbox{ and } P[g_n](z)= (k+1)!g_ng_n^{(k)}(z)=\frac{z^{k+2}}{4n^2}.$$ One can easily see that $f_n(z)=0\Rightarrow g_n(z)=0$ and $P[f_n](z)=1\not\Rightarrow P[g_n](z)=1.$ However, the family $\mathcal{F}$ is not normal in $\mathbb{D}.$

This example shows that the condition ``$P[f]$ and $P[g]$ partially share $a$ in $D$" in Theorem \ref{thm:2} is essential.
\end{example}

\begin{example}\label{exp:2}
Let $k$ be a positive integer and $a=1.$ Let $$\mathcal{F}=\left\{f_n:f_n(z)=nz^{k+1},~n\in\mathbb{N},~n\geq 2\right\}$$ and $$\mathcal{G}=\left\{g_n:g_n(z)=\left[z+1-\left(\frac{1}{n^2}\right)^{1/(k+2)}\right]^{k+1},~n\in\mathbb{N},~n\geq 2\right\}$$ be two families of holomorphic functions in $\mathbb{D}.$ Then each $f_n$ and $g_n$ has zeros of multiplicities at least $k+1$ and $g_n\rightarrow g=(z+1)^{k+1}\not\equiv\infty.$ Obviously, $f_n(z)=0\not\Rightarrow g_n(z)=0.$ Let $P[f]:=\frac{1}{(k+1)!}ff^{(k)}.$ Then $$P[f_n](z)=n^2z^{k+2}=1\Rightarrow z=\left(\frac{1}{n^{2}}\right)^{1/(k+2)}$$ and $$P[g_n](z)=\left[z+1-\left(\frac{1}{n^2}\right)^{1/(k+2)}\right]^{k+2}=1 \mbox{ whenever } z=\left(\frac{1}{n^2}\right)^{1/(k+2)}.$$ Thus $P[f_n](z)=1\Rightarrow P[g_n](z)=1.$ Note that the family $\mathcal{G}$ is normal in $\mathbb{D},$ however, the family $\mathcal{F}$ is not normal in $\mathbb{D}.$\\
Thus, the condition ``$f$ and $g$ partially share $0$  in $D$" can not be dropped.
\end{example}

\begin{example}\label{exp:3}
Let $k$ be a positive integer and $a$ be any non-zero complex number. Let $\mathcal{F}=\left\{f_n:n\in\mathbb{N}\right\}$ and $\mathcal{G}=\left\{g_n:n\in\mathbb{N}\right\}$ be two families of holomorphic functions in $\mathbb{D},$ where $$f_n(z)=\frac{e^{nz}}{n^k} \mbox{ and } g_n(z)=\frac{e^{nz}}{n^k}+ e^{2kn}.$$ Then each $f_n$ and $g_n$ omit zero in $\mathbb{D}$ and hence $f_n$ and $g_n$ partially share $0$  in $\mathbb{D}.$ Also, $g_n\rightarrow g\equiv\infty$ and so the family $\mathcal{G}$ is normal in $\mathbb{D}.$ Let $P[f]:=f'+f''+\cdots+ f^{(k)}.$ Then clearly $P[f_n](z)=P[g_n](z)$ and hence $P[f_n]$ and $P[g_n]$ partially share $a$ in $\mathbb{D}.$ However, the family $\mathcal{F}$ is not normal in $\mathbb{D}.$\\
This shows that the condition ``$g\not\equiv\infty$" can not be dropped.
\end{example}

\begin{example}\label{exp:4}
Let $k$ be a positive integer and $a$ be a non-zero complex number. Let $\mathcal{F}=\left\{f_n:f_n(z)=e^{nz}/n^k,~n\in\mathbb{N}\right\}$ and $\mathcal{G}=\left\{g_n:g_n(z)=e^z,~n\in\mathbb{N}\right\}$ 
be two families of holomorphic functions in $\mathbb{D}.$ Then $f_n$ and $g_n$ omit zero. Clearly, the family $\mathcal{G}$ is normal in $\mathbb{D}$ and $g_n\rightarrow g\equiv e^z\not\equiv\infty.$ Let $P[f]:=ae^{-z}f^{(k)}.$ Then $P[g]\equiv a.$ Further, $P[f_n](z)=ae^{-z}f_n^{(k)}(z)=ae^{z(n-1)}$ and $P[g_n](z)=ae^{-z}g_n^{(k)}(z)= a.$ It follows that $f_n(z)=0\Rightarrow g_n(z)=0$ and $P[f_n](z)=a\Rightarrow P[g_n](z)=a.$ However, the family $\mathcal{F}$ is not normal in $\mathbb{D}.$\\
This shows that the condition ``$P[g]\not\equiv a$" cannot be removed.
\end{example}

\begin{example}\label{exp:5}
Let $m,~k$ be two positive integers with $k\geq 2$ and let $a$ be a non-zero complex number. Let $\mathcal{F}=\left\{f_n:n\in\mathbb{N}\right\}$ and $\mathcal{G}=\left\{g_n:n\in\mathbb{N}\right\}$ be two families of holomorphic functions in $\mathbb{D},$ where $$f_n(z)=\frac{z^{k+1}}{n(k+1)}+nz \mbox{ and } g_n(z)=\frac{z^{k+1}}{n(k+1)}.$$ Then $f_n$ has only simple zeros and $g_n$  has zeros of multiplicity $k+1.$ Also, $f_n$ and $g_n$ partially share $0$  in $\mathbb{D}$ and $g_n\rightarrow g\equiv 0$ and so the family $\mathcal{G}$ is normal in $\mathbb{D}.$ Let $P[f]:=\left(f^{(k)}\right)^m.$ Then clearly $P[f_n](z)=P[g_n](z)$ and hence $P[f_n]$ and $P[g_n]$ partially share  $a$ in $\mathbb{D}.$ However, the family $\mathcal{F}$ is not normal in $\mathbb{D}.$\\
Thus the condition``all zeros of each $f\in\mathcal{F}$ have multiplicities at least $k+1$" cannot be relaxed.
\end{example}

\begin{example}\label{exp:6}
Let $k\geq 2$ be a positive integer and $a$ be a non-zero complex number. Let $\mathcal{F}=\left\{f_n:n\in\mathbb{N},~n\geq 2\right\}$ and $\mathcal{G}=\left\{g_n:n\in\mathbb{N},~n\geq 2\right\}$ be two families of holomorphic functions in $\mathbb{D}$ given by $$f_n(z)=\frac{nz^{k+1}}{(k+1)!} \mbox{ and } g_n(z)=\frac{z^{k+1}}{(k+1)!}+\frac{az^k}{k!}\left(1-\frac{1}{n}\right).$$ Then for each $n,$ $f_n$ has zeros of multiplicity $k+1$ and $g_n$ has zeros of multiplicity $k.$ Also, $f_n$ and $g_n$ partially share $0$ in $\mathbb{D}$ and $g_n\rightarrow g\equiv z^{k+1}/(k+1)!+ az^k/k!$ and so the family $\mathcal{G}$ is normal in $\mathbb{D}.$ Let $P[f]:=f^{(k)}.$ Then $P[g](z)=z+a\not\equiv a.$ Now, $P[f_n](z)=nz=a\Rightarrow z=a/n$ and $P[g_n](z)=z+a(1-1/n)=a~\mbox{whenever } z=a/n.$ Thus $P[f_n](z)$ and $P[g_n](z)$ partially share $a$ in $\mathbb{D}.$ However, the family $\mathcal{F}$ is not normal in $\mathbb{D}.$\\
This shows that the condition ``all zeros of each $g\in\mathcal{G}$ have multiplicities at least $k+1$" cannot be dropped.
\end{example}

\begin{example}\label{exp:7}
Let $a$ be any non-zero complex number. Consider the families $\mathcal{F}=\left\{f_n:f_n(z)=e^{-nz},~n\in\mathbb{N}\right\}$ and $\mathcal{G}=\left\{g_n:g_n(z)=e^z,~n\in\mathbb{N}\right\}$ of holomorphic functions in $\mathbb{D}.$ Then for each $n,$ $f_n$ and $g_n$ omit zero in $\mathbb{D}.$ Also, $g_n\rightarrow g\equiv e^z\not\equiv\infty$ and so the family $\mathcal{G}$ is normal in $\mathbb{D}.$ Let $P[f]:=M_1[f]+M_2[f],$ where $M_1[f]:=-f'f^{(4)}$ and $M_2[f]:=f''f^{(3)}$ so that $\Gamma_{M_1}/\gamma_{M_1}=7/2=\Gamma_{M_2}/\gamma_{M_2}.$ Then $P[g]\equiv 0\not\equiv a$ and $$P[f_n](z)=\frac{n}{e^{nz}}\cdot\frac{n^4}{e^{nz}}+\frac{n^2}{e^{nz}}\cdot\frac{-n^3}{e^{nz}}=0.$$ It follows vacuously that $f_n$ and $g_n$ partially share $0$  in $\mathbb{D}$ and $P[f_n]$ and $P[g_n]$ partially share $a$  in $\mathbb{D}.$ However, the family $\mathcal{F}$ is not normal in $\mathbb{D}.$\\
This shows that the condition ``$P[f]$ satisfies \eqref{eqn:2}" cannot be removed.
\end{example}

The meromorphic analogue of Theorem \ref{thm:2} is obtained as:

\begin{theorem}\label{thm:3}
Let $\mathcal{F},~\mathcal{G}\subset\mathcal{M}(D)$ be two families, all of whose zeros have multiplicities at least $k+1,~k\in\mathbb{N}.$ Let $P[f]$ be a differential polynomial defined in \eqref{eqn:1} and satisfying \eqref{eqn:2} such that $\Theta_P> 1.$ Let $a$ be a non-zero complex number. Assume that the family $\mathcal{G}$ is normal in $D$ such that for any subsequence $\left\{g_n\right\}\subset\mathcal{G},~g_n\rightarrow g,~g\not\equiv\infty$ and $P[g]\not\equiv a.$ If for every $f\in\mathcal{F},$ there exists $g\in\mathcal{G}$ such that $f$ and $g$ partially share $0$ and $\infty$, and  $P[f]$ and $P[g]$ parially share $a$  in $D,$ then $\mathcal{F}$ is normal in $D.$
\end{theorem}

Examples \ref{exp:1}-\ref{exp:7} along with the following examples show that all the conditions in Theorem \ref{thm:3} are essential:

\begin{example}
Let $k$ be a positive integer and $a=1.$ Let $\mathcal{F}=\left\{f_n:n\in\mathbb{N},~n\geq 2\right\}$ and $\mathcal{G}=\left\{g_n:n\in\mathbb{N},~n\geq 2\right\}$ be two families of meromorphic functions in $\mathbb{D},$ where $$f_n(z)=\frac{1}{(-1)^kk!~nz} \mbox{ and } g_n(z)=\frac{1}{(k+1)!}\left[z+1-\left(\frac{1}{n}\right)^{1/(k+1)}\right]^{k+1} .$$ Then for each $n,$ $f_n(z)\neq 0$ and $g_n$ has zeros of multiplicity $k+1.$ Obviously, $f_n$ and $g_n$ partially share $0$  in $\mathbb{D}.$ Also, $g_n\rightarrow g=1/(k+1)!(z+1)^{k+1}\not\equiv\infty$ and so the family $\mathcal{G}$ is normal in $\mathbb{D}.$ Let $P[f]:=f^{(k)}.$ Then $P[g]=z+1\not\equiv 1.$ Also, $$P[f_n](z)=\frac{1}{nz^{k+1}}=1\Rightarrow z=\left(\frac{1}{n}\right)^{1/(k+1)}$$ and $$P[g_n](z)=z+1-\left(\frac{1}{n}\right)^{1/(k+1)}=1 \mbox{ whenever } z=\left(\frac{1}{n}\right)^{1/(k+1)}.$$ Thus $P[f_n]$ and $P[g_n]$ partially share $1$  in $\mathbb{D}.$ However, the family $\mathcal{F}$ is not normal in $\mathbb{D}.$\\ 
This example shows that the condition ``$f$ and $g$ partially share $\infty$  in $D$" can not be dropped.
\end{example}

\begin{example}
Consider the families $$\mathcal{F}=\left\{f_n:f_n(z)=\frac{1}{e^{nz}+1},~n\in\mathbb{N}\right\}$$ and $$\mathcal{G}=\left\{g_n:g_n(z)=\frac{1}{e^n\left(e^{nz}+1\right)},~n\in\mathbb{N}\right\}$$ of meromorphic functions in $\mathbb{D}.$ Then  $f_n$ and $g_n$ omit zero. Clearly, $f_n$ and $g_n$ partially share $\infty$  in $\mathbb{D}.$ Also, $g_n\rightarrow g\equiv 0\not\equiv\infty$ and so the family $\mathcal{G}$ is normal in $\mathbb{D}.$ Let $P[f]:=f.$ Then $\Theta_P=1.$ Since $f_n$ omits $0$ and $1,$ it follows vacuously that $f_n$ and $g_n$ partially share $0$  in $\mathbb{D}$ and $P[f_n]$ and $P[g_n]$ partially share $1$  in $\mathbb{D}.$ However, the family $\mathcal{F}$ is not normal in $\mathbb{D}.$\\
This shows that the condition ``$\Theta_P>1$" cannot be removed.
\end{example}

\begin{remark}
The assumption that the coefficients of the differential polynomial $P[f]$ considered in Theorem \ref{thm:2} and Theorem \ref{thm:3} are holomorphic is essential. This can be seen from the following example:

Consider the families $\mathcal{F}=\left\{f_n:n\in\mathbb{N}\right\}$ and $\mathcal{G}=\left\{g_n:n\in\mathbb{N}\right\},$ where $$f_n(z)=\frac{nz^{k+1}}{(k+1)!} \mbox{ and } g_n(z)=\frac{z^{k+1}}{(k+1)!},~z\in\mathbb{D}.$$ Then $\mathcal{F},~\mathcal{G}\subset\mathcal{H}(\mathbb{D}),$ and $f_n$ and $g_n$ have zeros of multiplicities $k+1.$ Obviously, the family $\mathcal{G}$ is normal in $\mathbb{D}$ and $g_n\rightarrow g\equiv z^{k+1}/(k+1)!\not\equiv\infty.$ Let $P[f]:= (1/z)f^{(k)}.$ Then $P[f_n](z)=n \mbox{ and } P[g_n](z)=1.$ It follows that $f_n$ and $g_n$ partially share $0$ and $\infty$  in $\mathbb{D},$ and for any $a\in\mathbb{C}\setminus\mathbb{N},$ $P[f_n]$ and $P[g_n]$ partially share $a$ in $\mathbb{D}.$ However, the family $\mathcal{F}$ is not normal in $\mathbb{D}.$

Further, the assumption on the coefficients of the differential polynomial $P[f]$ namely, ``$a_1(z)\neq 0$" is essential. For example, consider the families $\mathcal{F}=\left\{f_n:n\in\mathbb{N}\right\}$ and $\mathcal{G}=\left\{g_n:n\in\mathbb{N}\right\},$ where $$f_n(z)=\frac{1}{nz} \mbox{ and } g_n(z)=\frac{1}{z},~z\in\mathbb{D}.$$ Then $\mathcal{F},~\mathcal{G}\subset\mathcal{M}(\mathbb{D}),$ and $f_n$ and $g_n$ omit zero. Obviously, the family $\mathcal{G}$ is normal in $\mathbb{D}$ and $g_n\rightarrow g\equiv 1/z\not\equiv\infty.$ Let $P[f]:= \frac{z^{k+1}}{(-1)^{k}k!}f^{(k)}.$ Then $$P[f_n](z)=\frac{z^{k+1}}{(-1)^{k}k!}\cdot\frac{(-1)^{k}k!}{nz^{k+1}}=\frac{1}{n} \mbox{ and } P[g_n](z)=1.$$ It follows that $f_n$ and $g_n$ partially share $0$ and $\infty$ in $\mathbb{D}$ and for any $a\in\mathbb{C}\setminus\left\{1/n:n\in\mathbb{N}\right\},$  $P[f_n]$ and $P[g_n]$ partially share $a$ in $\mathbb{D}.$ However, the family $\mathcal{F}$ is not normal in $\mathbb{D}.$       
\end{remark}

In \cite[Theorem 1.7]{singh}, Singh and the first author proved the following theorem:

\begin{theorem}\label{thm:singh}
Let $n_0,~n_1,\ldots,~n_k$ be non-negative integers such that $n_0\geq 2,~n_k\geq 1$ and $k\in\mathbb{N}.$ Let $h~(\not\equiv 0)$ be a holomorphic function in $D$ with zeros of multiplicity $m~(\geq 1).$ Let $\mathcal{F}\subset\mathcal{M}(D)$ be such that each $f\in\mathcal{F}$ has zeros of multiplicity at least $k+m+1$ and poles of multiplicity at least $m+1.$ If for every $f,~g\in\mathcal{F},$ $M[f]$ and $M[g]$ share $h$ in $D,$ then $\mathcal{F}$ is normal in $D.$
\end{theorem}

We remove the condition on the multiplicities of poles of $f\in\mathcal{F}$ and improve Theorem \ref{thm:singh} as:

\begin{theorem}\label{thm:1}
Let $n_0,~n_1,\ldots,~n_k$ be non-negative integers with $n_0\geq 2,~n_k\geq 1$ and $k\in\mathbb{N}.$ Let $h~(\not\equiv 0)$ be a holomorphic function in $D$ having zeros of multiplicity at most $m~(\geq 1).$ Let $\mathcal{F}\subset\mathcal{M}(D)$ be such that zeros of each $f\in\mathcal{F}$ have multiplicities at least $k+m+1.$ If for every $f,~g\in\mathcal{F},$ $M[f]$ and $M[g]$ share $h$ in $D,$ then $\mathcal{F}$ is normal in $D.$
\end{theorem}

The inevitability of the condition ``$h\not\equiv 0$" in Theorem \ref{thm:1} can be seen from the following example:

\begin{example}
Consider the family $\mathcal{F}=\left\{f_n:n\in\mathbb{N}\right\}$ of meromorphic functions on $\mathbb{D}$ given by $f_n(z)=e^{nz}.$ Then for each $n,$ $f_n$ omits $0.$ Let $M[f]:=f^2f'.$ Then $M[f_n](z)=f_n^2f_n'=ne^{3nz}.$ Clearly, for distinct $m,~n,$ $M[f_m]$ and $M[f_n]$ share $h\equiv 0$ in $\mathbb{D}.$ However, the family $\mathcal{F}$ is not normal in $\mathbb{D}.$ 
\end{example}

The following example demonstrates that the condition ``zeros of each $f\in\mathcal{F}$ have multiplicities at least $k+m+1$" in Theorem \ref{thm:1} cannot be dropped:

\begin{example}
Let $k\geq 2$ be a natural number and consider the family $\mathcal{F}=\left\{f_n:n\in\mathbb{N}\right\}$ of meromorphic functions on $\mathbb{D}$ given by $f_n(z)=nz^k.$ Then for each $f\in\mathcal{F}$ has zeros of multiplicity  $k.$ Let $M[f]:=f^2f'.$ Then $M[f_n](z)=f_n^2f_n'=kn^3z^{3k-1}.$ Take $h(z)=z^{3k-1}.$ Then $M[f_n](z)-h(z)=z^{3k-1}(kn^3-1).$ Clearly, for distinct $m,~n,$ $M[f_m]$ and $M[f_n]$ share $h$ in $\mathbb{D}.$ However, the family $\mathcal{F}$ is not normal in $\mathbb{D}.$ 
\end{example}

As an immediate consequence of Theorem \ref{thm:1}, we have

\begin{corollary}\label{cor:1}
Let $n_0,~n_1,\ldots,~n_k$ be non-negative integers with $n_0\geq 2,~n_k\geq 1$ and $k\in\mathbb{N}.$ Let $h$ be non-vanishing  holomorphic function in $D.$ Let $\mathcal{F}\subset\mathcal{M}(D)$ be such that for each $f\in\mathcal{F},$ $M[f]-h$ has no zero in $D.$ Then $\mathcal{F}$ is normal in $D.$
\end{corollary}

It is noteworthy to mention that Corollary \ref{cor:1} leads to a counterexample to the converse of Bloch's principle (see \cite{bergweiler, charak-1, charak-2, charak-3, lahiri, li}) which states that if a family $\mathcal{F}\subset\mathcal{M}(D)$ satisfying a certain property $\mathcal{P}$ in $D$ is normal, then every $f\in\mathcal{M}(\mathbb{C})$ which satisfies property $\mathcal{P}$ in $\mathbb{C}$ reduces to a constant.

\begin{counterexample}
Let $D$ be any domain in $\mathbb{C}$ and let $\mathcal{F}=\left\{f_n: n\in\mathbb{N}\right\}$ be a family of meromorphic functions in $D$ given by $f_n(z)=e^{z}.$ Let $\mathcal{P}$ be the property that for each $f\in\mathcal{F},$ $M[f]-h$ has no zero in $D,$ where $h$ is a non-vanishing holomorphic function in $D.$ In view of Corollary \ref{cor:1}, the family $\mathcal{F}$ is normal in $D.$ Now, let $f(z)=e^{z},~M[f]=f^2f'$ and $h(z)=-e^{3z}.$ Then $f\in\mathcal{H}(\mathbb{C}),~h\in\mathcal{H}(\mathbb{C})$ and $h(z)\neq 0.$ Clearly, $M[f](z)-h(z)=2e^{3z}\neq 0,~\forall~z\in\mathbb{C}.$ Then $f$ satisfies property $\mathcal{P}$ in $\mathbb{C}.$ However, $f$ is non-constant. This violates the statement of the converse of Bloch's principle.
\end{counterexample}

\section{Some Value Distribution Results and Preliminary Lemmas}
In this section, we state and prove some  results which are crucial to prove the main results of this paper. Our first preliminary result is an extension of the famous Zalcman-Pang Lemma due to Chen and Gu \cite{chen-gu} (see also \cite[p. 216]{zalcman-2}, cf. \cite[Lemma 2]{pang-zalcman}).

\begin{lemma}[Zalcman-Pang Lemma]\label{lem:zp}
Let $\mathcal{F}\subset\mathcal{M}(D)$ be such that all of its zeros have multiplicities at least $m$ and all its poles have multiplicities at least $p.$ Let $-p<\alpha<m.$ If $\mathcal{F}$ is not normal at $z_0\in D,$ then there exist
sequences $\left\{f_n\right\}\subset\mathcal{F},$ $\left\{z_n\right\}\subset D$ satisfying $z_n\rightarrow z_0$ and positive numbers $\rho_n$ with $\rho_n\rightarrow 0$ such that the sequence $\left\{g_n\right\}$ defined by $$g_n(\zeta)=\rho_{n}^{-\alpha}f_n(z_n+\rho_n\zeta)\rightarrow g(\zeta)$$ locally uniformly in $\mathbb{C}$ with respect to the spherical metric, where $g$ is a non-constant meromorphic function on $\mathbb{C}$ such that for every $\zeta\in\mathbb{C},$ $g^{\#}(\zeta)\leq g^{\#}(0)=1.$
 \end{lemma}

The following lemma is due to Hayman (see \cite[Theorem 3]{hayman-2}, cf. \cite[Theorem 3.5]{hayman}) commonly known as ``Hayman's Alternative": 

\begin{lemma}\label{lem:h}
Let $f\in\mathcal{M}(\mathbb{C})$ be a non-constant function and $k$ be a positive integer. Then  for any $a\in\mathbb{C}\setminus\left\{0\right\},$ either $f$ or $f^{(k)}-a$ has at least one zero in $\mathbb{C}.$ Moreover, if $f$ is transcendental, then $f$ or $f^{(k)}-a$ has infinitely many zeros in $\mathbb{C}.$ 
\end{lemma}

\begin{lemma}(\cite[Lemma 6]{wang}, cf. \cite[Theorem 3]{bergo})\label{lem:w1}
Let $f\in\mathcal{M}(\mathbb{C})$ be a transcendental function of finite order having zeros of multiplicity at least $k+1,$ where $k\in\mathbb{N}.$ Then for $j=1,\ldots,k,$ $f^{(j)}$ assumes every non-zero complex number infinitely many times. 
\end{lemma}

\begin{lemma}\label{lem:lo}\cite[Lemma 1.2]{lo}
Let $f_j~(j=1,~2)$ be two non-constant meromorphic functions in $\mathbb{C}.$ Then $$N(r,f_1f_2)-N\left(r,\frac{1}{f_1f_2}\right)=N(r,f_1)+N(r,f_2)-N\left(r,\frac{1}{f_1}\right)-N\left(r,\frac{1}{f_2}\right).$$
\end{lemma}

\begin{lemma}\label{singh}\cite[Lemma 2.1]{singh}
Let $n_0,n_1,\ldots,n_k, m$ be non-negative integers with $n_0\geq 2,~\sum\limits_{j=1}^{k}n_j\geq 1$ and $k\in\mathbb{N}.$ Let $p~(\not\equiv 0)$ be a polynomial of degree $m$ and $f$ be a non-constant rational function with zeros of multiplicity at least $k+m$ and poles of multiplicity at least $m+1.$ Let $M[f]=\prod\limits_{j=0}^{k}\left(f^{(j)}\right)^{n_j}$ be a differential monomial of $f.$ Then $M[f]-p$ has at least two distinct zeros in $\mathbb{C}.$
\end{lemma}

\begin{lemma}\label{lem:1}
Let $f$ be a non-constant and non-vanishing rational function in $\mathbb{C}$ and let $p~(\not\equiv 0)$ be a polynomial. Let $M[f]=\prod\limits_{j=0}^{k}\left(f^{(j)}\right)^{n_j}$ be a differential monomial of $f,$ where $k\in\mathbb{N}$ and $n_j$'s ($j=0,1,\ldots,k$) are non-negative integers such that $\sum\limits_{j=0}^{k}n_j\geq 1.$ Then $M[f]-p$ has at least $\Gamma_M$ distinct zeros in $\mathbb{C}.$  
\end{lemma}

\begin{proof}
The proof is based on the method of Chang \cite{chang} (see also \cite{deng}) with significant modifications. Since the computations involved in the proof are intricate, we describe the proof in detail.

Since $f$ is non-vanishing, $f$ cannot be a polynomial and so $f$ has at least one pole. Thus we can write
\begin{equation}\label{eq:l1.1}
f(z)=\frac{A}{\prod_{i=1}^{t}(z+z_i)^{t_i}}
\end{equation}
and 
\begin{equation}\label{eq:l1.2}
p(z)=B\prod\limits_{i=1}^{d}(z+ v_i)^{d_i},
\end{equation}
where $A,~B$ are non-zero constants, $d,~t,~t_i$ are positive integers and $d_i$ are non-negative integers. Also, $v_i$ (when $1\leq i\leq d$) are distinct complex numbers and $z_i$ (when $1\leq i\leq t$) are distinct complex numbers.

From \eqref{eq:l1.1}, we deduce that 
\begin{equation}\label{eq:l1.3}
M[f](z)=\frac{g(z)}{\prod_{i=1}^{t}(z+z_i)^{(t_i-1)\gamma_M + \Gamma_M}},
\end{equation}
where  $g$ is a polynomial of degree $(t-1)(\Gamma_M-\gamma_M).$

Also, by assumption, it easily follows that $M[f]-p$ has at least one zero in $\mathbb{C}.$ Thus we may set
\begin{equation}\label{eq:l1.4}
M[f](z)=p(z)+\frac{C\prod\limits_{i=1}^{s}(z+w_i)^{l_i}}{\prod_{i=1}^{t}(z+z_i)^{(t_i-1)\gamma_M + \Gamma_M}},
\end{equation}
where $C(\neq 0)\in\mathbb{C},$ $l_i$ are positive integers and $w_i$ ($1\leq i\leq s$) are distinct complex numbers.

Let $D=\sum\limits_{i=1}^{d}d_i$ and $T=\sum\limits_{i=1}^{t}t_i.$ Then from \eqref{eq:l1.2}, \eqref{eq:l1.3} and \eqref{eq:l1.4}, we get
\begin{equation}\label{eq:l1.5}
B\prod\limits_{i=1}^{d}(z+v_i)^{d_i}\prod_{i=1}^{t}(z+z_i)^{(t_i-1)\gamma_M + \Gamma_M} + C\prod\limits_{i=1}^{s}(z+w_i)^{l_i}=g(z)
\end{equation}
By simple calculation, one can easily see that $$\sum\limits_{i=1}^{s}l_i=\sum\limits_{i=1}^{t}\left[(t_i-1)\gamma_M + \Gamma_M\right] + \sum\limits_{i=1}^{d}d_i=(T-t)\gamma_M +t\Gamma_M + D$$ and $C=-B.$

Also, from \eqref{eq:l1.5}, we get

$$\prod\limits_{i=1}^{d}(1+v_ir)^{d_i}\prod_{i=1}^{t}(1+z_ir)^{(t_i-1)\gamma_M + \Gamma_M} - \prod\limits_{i=1}^{s}(1+w_ir)^{l_i}=r^{\Gamma_M+(T-1)\gamma_M + D}h(r),$$ where $h(r):=r^{(t-1)(\Gamma_M-\gamma_M)}g(1/r)/B$ is a polynomial of degree less than $(t-1)(\Gamma_M-\gamma_M).$ Furthermore, it follows that
\begin{align}\label{eq:l1.6}
\frac{\prod\limits_{i=1}^{d}(1+v_ir)^{d_i}\prod\limits_{i=1}^{t}(1+z_ir)^{(t_i-1)\gamma_M + \Gamma_M}}{\prod\limits_{i=1}^{s}(1+w_ir)^{l_i}}&=1+\frac{r^{\Gamma_M+(T-1)\gamma_M + D}h(r)}{\prod\limits_{i=1}^{s}(1+w_ir)^{l_i}}\nonumber\\ 
 &=1+ O\left(r^{\Gamma_M+(T-1)\gamma_M + D}\right) 
\end{align}
as $r\rightarrow 0.$ Taking logarithmic derivatives of both sides of \eqref{eq:l1.6}, we obtain
\begin{equation}\label{eq:l1.7}
\sum\limits_{i=1}^{d}\frac{d_iv_i}{1+v_ir}+ \sum\limits_{i=1}^{t}\frac{\left[(t_i-1)\gamma_M + \Gamma_M\right]z_i}{1+z_ir}-\sum\limits_{i=1}^{s}\frac{l_iw_i}{1+w_ir}=O\left(r^{\Gamma_M+(T-1)\gamma_M + D-1}\right)
\end{equation}
as $r\rightarrow 0.$

Let $S_1=\left\{v_1,v_2,\ldots,v_d\right\}\cap\left\{z_1,z_2,\ldots,z_t\right\}$ and $S_2=\left\{v_1,~v_2,\ldots,v_d\right\}\cap\left\{w_1,w_2,\ldots,w_s\right\}.$ Then the following four cases arise:\\

{\bf  Case 1:} $S_1=S_2=\phi$. Let $z_{t+i}=v_i$ when $1\leq i\leq d$ and $$T_i=\left\{\begin{array}{cc} (t_i-1)\gamma_M + \Gamma_M & \mbox{if}~ 1\leq i\leq t,\\ d_{i-t} & \mbox{if}~ t+1\leq i\leq t+d.\end{array}\right.$$ Then \eqref{eq:l1.7} can be written as 
\begin{equation}\label{eq:l1.8}
\sum\limits_{i=1}^{t+d}\frac{T_iz_i}{1+z_ir}-\sum\limits_{i=1}^{s}\frac{l_iw_i}{1+w_ir}=O\left(r^{\Gamma_M+(T-1)\gamma_M + D-1}\right)~\mbox{as $r\rightarrow 0.$}
\end{equation}
Comparing the coefficients of $r^j,~j=0,1,\ldots,\Gamma_M+(T-1)\gamma_M + D-2$ in \eqref{eq:l1.8}, we find that 
\begin{equation}\label{eq:l1.9}
\sum\limits_{i=1}^{t+d}T_iz^j-\sum\limits_{i=1}^{s}l_iw^j=0,~\mbox{for each}~ j=1,2,\ldots,\Gamma_M+(T-1)\gamma_M + D-1.
\end{equation}
Let $z_{t+d+i}=w_i,~1\leq i\leq s.$ Then from \eqref{eq:l1.9}, we deduce that the system of linear equations
\begin{equation}\label{eq:l1.10}
\sum\limits_{i=1}^{t+d+s}z_i^jx_i=0,~j=0,1,\ldots,\Gamma_M+(T-1)\gamma_M + D-1,
\end{equation}
has a non-zero solution $$\left(x_1,\ldots,x_{t+d},x_{t+d+1},\ldots,x_{t+d+s}\right)=\left(T_1,\ldots,T_{t+d},-l_1,\ldots,-l_s\right).$$
If $\Gamma_M+(T-1)\gamma_M + D\geq t+d+s,$ then by Cramer's rule (see \cite[p. 134]{mirsky}), for $0\leq j\leq\Gamma_M+(T-1)\gamma_M + D-1,$ the determinant det$(z_i^j)_{(t+d+s)\times (t+d+s)}$ of the coefficients of the system \eqref{eq:l1.10} is equal to zero. However, since $z_i,~1\leq i\leq t+d+s$ are distinct and det$(z_i^j)_{(t+d+s)\times (t+d+s)}$ is a Vandermonde determinant (see \cite[p. 17]{mirsky}), so it cannot be equal to zero, a contradiction. Hence $\Gamma_M+(T-1)\gamma_M + D< t+d+s.$ Since $T=\sum\limits_{i=1}^{t}t_i\geq t$ and $D=\sum\limits_{i=1}^{d}d_i\geq d,$ it follows that $s\geq\Gamma_M.$ \\

{\bf  Case 2:} $S_1\neq\phi$ and $S_2=\phi.$ We may assume, without loss of generality, that $S_1=\left\{v_1,v_2,\ldots,v_{m_1}\right\}.$ Then $v_i=z_i$ for $1\leq i\leq m_1.$ Take $m_3=d-m_1.$\\

{\bf  Subcase 2.1:} $m_3\geq 1.$ Let $z_{t+i}=v_{m_1+i}$ for $1\leq i\leq m_3$ and if $m_1<t,$ then set $$T_i=\left\{\begin{array}{ccc} (t_i-1)\gamma_M + \Gamma_M+d_i & \mbox{if } 1\leq i\leq m_1,\\  (t_i-1)\gamma_M + \Gamma_M & \mbox{if } m_1+1\leq i\leq t,\\  d_{m_1-t+i} & \mbox{if } t+1\leq i\leq t+m_3.\end{array}\right.$$ If $m_1=t,$ then set $$T_i=\left\{\begin{array}{cc} (t_i-1)\gamma_M + \Gamma_M+d_i & \mbox{if } 1\leq i\leq m_1=t,\\ d_{m_1-t+i} & \mbox{if } t+1\leq i\leq t+m_3.\end{array}\right.$$

{\bf  Subcase 2.2:} $m_3=0.$ If $m_1<t,$ then set $$T_i=\left\{\begin{array}{cc}(t_i-1)\gamma_M + \Gamma_M+d_i & \mbox{if } 1\leq i\leq m_1,\\ (t_i-1)\gamma_M + \Gamma_M & \mbox{if } m_1+1\leq i\leq t\end{array}\right.$$ and if $m_1=t,$ then set $T_i=(t_i-1)\gamma_M + \Gamma_M+d_i,~\mbox{for } 1\leq i\leq m_1=t.$ 

Thus \eqref{eq:l1.7} can be written as: $$\sum\limits_{i=1}^{t+m_3}\frac{T_iz_i}{1+z_ir}-\sum\limits_{i=1}^{s}\frac{l_iw_i}{1+w_ir}=O\left(r^{\Gamma_M+(T-1)\gamma_M + D-1}\right)~\mbox{as $r\rightarrow 0,$}$$ where $0\leq m_3\leq d-1.$ Proceeding in similar fashion as in Case 1, we deduce that $s\geq\Gamma_M.$\\

{\bf  Case 3:} $S_1=\phi$ and $S_2\neq\phi.$ We may assume, without loss of generality, that $S_2=\left\{v_1,v_2,\ldots,v_{m_2}\right\}.$ Then $v_i=w_i$ for $1\leq i\leq m_2.$ Take $m_4=d-m_2.$\\

{\bf  Subcase 3.1:} $m_4\geq 1.$ Let $w_{s+i}=v_{m_2+i}$ for $1\leq i\leq m_4$ and if $m_2<s,$ then set $$L_i=\left\{\begin{array}{ccc} l_i-d_i & \mbox{if } 1\leq i\leq m_2,\\  l_i & \mbox{if } m_2+1\leq i\leq s,\\  -d_{m_2-s+i} & \mbox{if } s+1\leq i\leq s+m_4.\end{array}\right.$$ If $m_2=s,$ then set $$L_i=\left\{\begin{array}{cc} l_i-d_i & \mbox{if } 1\leq i\leq m_2=s,\\ -d_{m_2-s+i} & \mbox{if } s+1\leq i\leq s+m_4.\end{array}\right.$$

{\bf  Subcase 3.2:} $m_4=0.$ If $m_2<s,$ then set $$L_i=\left\{\begin{array}{cc}l_i-d_i & \mbox{if } 1\leq i\leq m_2,\\ l_i & \mbox{if } m_2+1\leq i\leq s\end{array}\right.$$ and if $m_2=s,$ then set $L_i=l_i-d_i,~\mbox{for } 1\leq i\leq m_2=s.$ 

Thus \eqref{eq:l1.7} can be written as: $$\sum\limits_{i=1}^{t}\frac{\left[(t_i-1)\gamma_M + \Gamma_M\right]z_i}{1+z_ir}-\sum\limits_{i=1}^{s+m_4}\frac{L_iw_i}{1+w_ir}=O\left(r^{\Gamma_M+(T-1)\gamma_M + D-1}\right)~\mbox{as $r\rightarrow 0,$}$$ where $0\leq m_4\leq d-1.$ Proceeding in similar way as in Case 1, we deduce that $s\geq\Gamma_M.$\\

{\bf  Case 4.} $S_1\neq\phi$ and $S_2\neq\phi.$ We may assume, without loss of generality, that $S_1=\left\{v_1,v_2,\ldots,v_{m_1}\right\},~S_2=\left\{v_1,v_2,\ldots,v_{m_2}\right\}.$ Then $v_i=z_i$ for $1\leq i\leq m_1$ $w_i=v_{m_1+i}$ for $1\leq i\leq m_2.$ Take $m_5=d-m_2-m_1.$\\

{\bf  Subcase 4.1:} $m_5\geq 1.$ Let $z_{t+i}=v_{m_1+m_2+i}$ for $1\leq i\leq m_5$ and if $m_1<t,$ then set $$T_i=\left\{\begin{array}{ccc} (t_i-1)\gamma_M + \Gamma_M+d_i & \mbox{if } 1\leq i\leq m_1,\\  (t_i-1)\gamma_M + \Gamma_M & \mbox{if } m_1+1\leq i\leq t,\\  d_{m_1+m_2-t+i} & \mbox{if } t+1\leq i\leq t+m_5.\end{array}\right.$$ If $m_1=t,$ then set $$T_i=\left\{\begin{array}{cc} (t_i-1)\gamma_M + \Gamma_M+d_i & \mbox{if } 1\leq i\leq m_1=t,\\ d_{m_1+m_2-t+i} & \mbox{if } t+1\leq i\leq t+m_5.\end{array}\right.$$ If $m_2<s,$ then set $$L_i=\left\{\begin{array}{cc}l_i-d_{m_1+i} & \mbox{if } 1\leq i\leq m_2,\\ l_i & \mbox{if } m_2+1\leq i\leq s\end{array}\right.$$ and if $m_2=s,$ then set $L_i=l_i-d_{m_1+i},~\mbox{for } 1\leq i\leq m_2=s.$\\

{\bf  Subcase 4.2:} $m_5=0.$ If $m_1<t,$ then set $$T_i=\left\{\begin{array}{cc} (t_i-1)\gamma_M + \Gamma_M+d_i & \mbox{if } 1\leq i\leq m_1,\\  (t_i-1)\gamma_M + \Gamma_M & \mbox{if } m_1+1\leq i\leq t.\end{array}\right.$$ If $m_1=t,$ then set $T_i=(t_i-1)\gamma_M + \Gamma_M+d_i~\mbox{for } 1\leq i\leq m_1=t.$ 

Also, if $m_2<s,$ then set $$L_i=\left\{\begin{array}{cc}l_i-d_{m_1+i} & \mbox{if } 1\leq i\leq m_2,\\ l_i & \mbox{if } m_2+1\leq i\leq s\end{array}\right.$$ and if $m_2=s,$ then set $L_i=l_i-d_{m_1+i},~\mbox{for } 1\leq i\leq m_2=s.$

Then in both subcases, \eqref{eq:l1.7} can be written as $$\sum\limits_{i=1}^{t+m_5}\frac{T_iz_i}{1+z_ir}-\sum\limits_{i=1}^{s}\frac{L_iw_i}{1+w_ir}=O\left(r^{\Gamma_M+(T-1)\gamma_M + D-1}\right)~\mbox{as $r\rightarrow 0,$}$$ where $0\leq m_5\leq d-2.$ Proceeding in similar fashion as in Case 1, we deduce that $s\geq\Gamma_M.$ This completes the proof.
\end{proof}

The following lemma is an extension of a result of Deng et al \cite[Lemma 5]{deng-2} to a more general differential monomial:

\begin{lemma}\label{lem:2}
Let $k$ be a positive integer and $n_0,~n_1,\ldots,~n_k$ be non-negative integers with $n_0\geq 2,~n_k\geq 1.$ Let $f$ be a non-constant meromorphic function in $\mathbb{C}$ and $p$ be a polynomial of degree $m\geq 1.$ Let $M[f]=\prod\limits_{j=0}^{k}\left(f^{(j)}\right)^{n_j}$ be a differential monomial of $f.$ If $f$ has zeros of multiplicity at least $k+m,$ then $M[f]-p$ has at least two distinct zeros in $\mathbb{C}.$ 
\end{lemma}

\begin{proof}
Since $f\not\equiv 0,$ we can write $$\frac{1}{f^{\gamma_M}}=\frac{M[f]}{pf^{\gamma_M}}-\frac{pM[f]'-p'M[f]}{p~f^{\gamma_M}}\cdot\frac{M[f]-p}{pM[f]'-p'M[f]},$$ where $\gamma_M=n_0+n_1+\cdots+n_k.$ 

Then
\begin{align*}
\gamma_M~m\left(r,\frac{1}{f}\right) &\leq m\left(r,\frac{M[f]}{pf^{\gamma_M}}\right)+ m\left(r, \frac{p~M[f]'-p'M[f]}{pf^{\gamma_M}}\right)\\ &\qquad
 + m\left(r, \frac{M[f]-p}{p~M[f]'-p'M[f]}\right) + O(1).
\end{align*}
From the Nevanlinna's theorem on logarithmic derivative, we have $m\left(r,\frac{M[f]}{f^{\gamma_M}}\right)=S(r,f).$ Also, using $m(r,p)=m\log{r} + O(1),~m\left(r,\frac{1}{p}\right)=O(1)$ and Lemma \ref{lem:lo}, we get
\begin{align*}
\gamma_M~m\left(r,\frac{1}{f}\right) &\leq m\left(r, \frac{M[f]-p}{p~M[f]'-p'M[f]}\right)+ S(r,f)\\
&= m\left(r, \frac{p~M[f]'-p'M[f]}{M[f]-p}\right) + N\left(r, \frac{p~M[f]'-p'M[f]}{M[f]-p}\right)\\ &\qquad
-~N\left(r, \frac{M[f]-p}{p~M[f]'-p'M[f]}\right) + S(r,f)\\
&= m\left(r, \frac{p\left[\frac{M[f]}{p}-1\right]'}{\left[\frac{M[f]}{p}-1\right]}\right) + N\left(r, p~M[f]'-p'M[f]\right) + N\left(r. \frac{1}{M[f]-p}\right)\\ &\qquad -~N\left(r, \frac{1}{p~M[f]'-p'M[f]}\right)- N\left(r, M[f]-p\right) + S(r,f)\\
&\leq \overline{N}(r,f)+ N\left(r,\frac{1}{M[f]-p}\right)-N\left(r,\frac{1}{p~M[f]'-p'M[f]}\right)\\ &\qquad + m\log{r}+ S(r,f).
\end{align*}
This gives
\begin{align}\label{eq:1}
\gamma_M~T(r,f)&\leq \overline{N}(r,f)+ \gamma_M~N\left(r,\frac{1}{f}\right)+ N\left(r,\frac{1}{M[f]-p}\right)\nonumber\\ 
&\qquad - N\left(r,\frac{1}{pM[f]'-p'M[f]}\right)+ m\log{r}+ S(r,f).
\end{align}
Note that if $z_0$ is a zero of $f$ with multiplicity $l$ ($\geq k+m$), then $z_0$ is a zero of $pM[f]'-p'M[f]$ with multiplicity at least $\gamma_Ml-(\Gamma_M-\gamma_M)-1,$ where $\Gamma_M-\gamma_M=n_1+2n_2+\cdots+ kn_k.$ Therefore, $$N\left(r,\frac{1}{p~M[f]'-p'M[f]}\right)\geq \left[\gamma_Ml-(\Gamma_M-\gamma_M)-1\right]\overline{N}\left(r,\frac{1}{f}\right).$$ Also, since $pM[f]'-p'M[f]=p\left(M[f]-p\right)'-p'\left(M[f]-p\right),$ we have $$N\left(r,\frac{1}{p~M[f]'-p'M[f]}\right)\geq N\left(r,\frac{1}{M[f]-p}\right)-\overline{N}\left(r,\frac{1}{M[f]-p}\right).$$
Thus from \eqref{eq:1}, we obtain
\begin{align}\label{eq:2}
\gamma_M~T(r,f)&\leq \overline{N}(r,f)+ (\Gamma_M-\gamma_M+1)\overline{N}\left(r,\frac{1}{f}\right) + \overline{N}\left(r,\frac{1}{M[f]-p]}\right)\nonumber\\ &\qquad +~m\log{r}+ S(r,f)\nonumber\\ 
 &\leq \overline{N}(r,f)+ (\Gamma_M-\gamma_M+1)\overline{N}\left(r,\frac{1}{f}\right) + \overline{N}\left(r,\frac{1}{M[f]-p]}\right)\nonumber\\ &\qquad + m\log{r} + S(r,f)\nonumber\\
&\leq\overline{N}(r,f)+ \frac{\Gamma_M-\gamma_M+1}{k+m}~N\left(r,\frac{1}{f}\right) + \overline{N}\left(r,\frac{1}{M[f]-p]}\right)\nonumber\\ &\qquad + m\log{r} + S(r,f).
\end{align}
Thus $$\left[(\gamma_M-1)-\frac{\Gamma_M-\gamma_M+1}{k+m}\right]T(r,f) \leq \overline{N}\left(r,\frac{1}{M[f]-p]}\right) + m\log{r}+ S(r,f).$$

That is, \begin{equation}\label{eq:3} C~T(r,f)\leq \overline{N}\left(r,\frac{1}{M[f]-p]}\right) + m\log{r}+ S(r,f),\end{equation}
where \begin{align*}
C &:= (\gamma_M-1)-\frac{\Gamma_M-\gamma_M+1}{k+m}\\
&= (n_0 + n_1+ n_2 + \cdots+n_k-1)-\frac{(n_1+2n_2+\cdots+ kn_k)+1}{k+m}\\
&= n_0-1 + \frac{(k+m-1)n_1 + (k+m-2)n_2 + \cdots + (m+1)n_{k-1} + mn_k-1}{k+m}\\
&= n_0-1 + \frac{\alpha +mn_k-1}{k+m},
\end{align*}
wherein $\alpha:=(k+m-1)n_1 + (k+m-2)n_2 + \cdots + (m+1)n_{k-1}\geq 0.$

Now suppose that $M[f]-p$ has at most one distinct zero. Then we consider the following cases:

{\bf  Case 1:} $\alpha\neq 0.$ Then $$C\geq 1 + \frac{\alpha}{k+m}$$ and so from \eqref{eq:3}, we obtain $$T(r,f)<C~T(r,f)\leq (m+1)\log{r}+ S(r,f).$$ This implies that $f$ is a rational function of degree at most $m.$ However, zeros of $f$ have multiplicities at least $k+m\geq m+1.$ Hence $f$ must be non-vanishing and so by Lemma \ref{lem:1}, $M[f]-p$ has at least $n_0 + 2n_1+\cdots+(k+1)n_k\geq 4$ distinct zeros, a contradiction.

{\bf  Case 2:} $\alpha=0.$ Then $$C=n_0-1+\frac{mn_k-1}{k+m}.$$ If either $n_k\geq 2$ or $m\geq 2,$ then $$C\geq 1+\frac{1}{k+m}$$ and so by the same argument as in Case 1, we get a contradiction. Thus we only need to consider the case when $n_k=1$ and $m=1.$ So, let $n_k=m=1.$ Then from \eqref{eq:3}, we obtain $$(n_0-1)T(r,f)\leq \overline{N}\left(r,\frac{1}{M[f]-p]}\right) + \log{r}+ S(r,f).$$ If $M[f](z)-p(z)\neq 0,$ then $$T(r,f)\leq (n_0-1)T(r,f)\leq\log{r}+ S(r,f),$$ showing that $f$ is a rational function of degree at most one. Since zeros of $f$ have multiplicities at least $k+1\geq 2,$ it follows that $f$ must be non-vanishing and so by Lemma \ref{lem:1}, $M[f]-p$ has at least $n_0+2n_1+\cdots+(k+1)n_k\geq 4$ distinct zeros, a contradiction. Thus by the assumption, $M[f]-p$ has exactly one zero. Then from \eqref{eq:2}, we have  
\begin{equation}\label{eq:4}
n_0T(r,f)\leq \overline{N}(r,f)+ 2\log{r}+S(r,f).
\end{equation}
{\bf  Subcase 2.1:} $n_0\geq 3.$ Then from \eqref{eq:4}, we have $$T(r,f)\leq\log{r} + S(r,f),$$ showing that $f$ is a rational function of degree at most one. Again, since zeros of $f$ have multiplicities at least $k+1\geq 2,$ $f$ must be non-vanishing and so by Lemma \ref{lem:1}, $M[f]-p$ has at least $n_0+2n_1+\cdots+(k+1)n_k\geq 5$ distinct zeros, a contradiction.\\

{\bf  Subcase 2.2:} $n_0=2.$ Then \eqref{eq:4} gives $$T(r,f)\leq 2\log{r} + S(r,f).$$ This means that $f$ is a rational function of degree at most $2.$ If $k\geq 2,$ then the fact that $f$ has zeros of multiplicity at least $k+1\geq 3$ implies that $f$ is non-vanishing and hence by Lemma \ref{lem:1}, $M[f]-p$ has at least $n_0+2n_1+\cdots+(k+1)n_k\geq 5$ distinct zeros, a contradiction. Hence $k=1.$

Now since $f$ is a rational function of degree at most $2,$ it follows that either $f$ is non-vanishing or $f$ has only one distinct zero with multiplicity $2.$ If $f(z)\neq 0,$ then by Lemma \ref{lem:1}, $M[f]-p$ has at least $n_0+2n_1+\cdots+(k+1)n_k\geq 4$ distinct zeros, a contradiction. Hence $f$ has the following forms:
$$A_1: f(z)=a(z-z_0)^2,~~A_2: f(z)=\frac{a(z-z_0)^2}{z-z_1};$$
$$A_3: f(z)=\frac{a(z-z_0)^2}{(z-z_1)^2},~~ A_4: f(z)=\frac{a(z-z_0)^2}{(z-z_1)(z-z_2)},$$ where $a~(\neq 0)\in\mathbb{C}.$ Clearly, $T(r,f)=2\log{r}+ O(1).$ 

If $f$ is of the form $A_1,$ then $\overline{N}\left(r,f\right)=0$ and $\overline{N}\left(r,\frac{1}{f}\right)\leq 1/2~T(r,f) + O(1).$ Also, from \eqref{eq:2}, it follows that $$3~T(r,f)\leq\overline{N}(r,f)+ 2~\overline{N}\left(r,\frac{1}{f}\right) + 2\log{r}+ S(r,f).$$ This then implies that $T(r,f)\leq\log{r}+S(r,f),$ a contradiction. Similarly, if $f$ is of the form $A_2$ or $A_3$, then we obtain $T(r,f)\leq 4/3\log{r}+S(r,f),$ which is again a contradiction. 

Hence $$f(z)=\frac{a(z-z_0)^2}{(z-z_1)(z-z_2)}.$$
Then 
\begin{equation}\label{eq:5}
M[f](z)=f^2(z)f'(z)=\frac{a^3(z-z_0)^5[(2z_0-z_1-z_2)z+2z_1z_2-z_0(z_1+z_2)]}{(z-z_1)^4(z-z_2)^4}
\end{equation}
Since $p$ is a polynomial of degree $m=1,$ we set $p(z)=b(z-z_3),$ where $b~(\neq 0),~z_3\in\mathbb{C}.$ Also, since $M[f]-p$ has exactly one zero, we can write
\begin{equation}\label{eq:6}
M[f](z)=b(z-z_3)+\frac{d(z-z_4)^r}{(z-z_1)^4(z-z_2)^4}
\end{equation}
By simple calculation, one can easily see that $d=-b,~r=9$ and $z_4\neq z_0.$ Then from \eqref{eq:5}, we obtain 
\begin{equation}\label{eq:7}
M[f]''(z)=\frac{(z-z_0)^3g_1(z)}{(z-z_1)^6(z-z_2)^6},
\end{equation}
where $g_1$ is a polynomial of degree at most $5.$ 

Also, from \eqref{eq:6}, we get
\begin{equation}\label{eq:8}
M[f]''(z)=\frac{(z-z_4)^7g_2(z)}{(z-z_1)^6(z-z_2)^6},
\end{equation}
where $g_2$ is a polynomial of degree at most $4.$ 

Comparing \eqref{eq:7} and \eqref{eq:8} and using the fact that $z_4\neq z_0,$ we find that $g_1$ is a polynomial of degree at least $7,$ a contradiction. 
This completes the proof. 
\end{proof}

The following lemma is a generalization of Hayman's Alternative and Lemma \ref{lem:w1} to differential polynomials: 

\begin{lemma}\label{lem:3}
Let $f\in\mathcal{M}(\mathbb{C})$ be a transcendental function and $\psi~(\not\equiv 0,\infty)$ be a small function of $f$ (that is, $T(r,\psi)=S(r,f)$ as $r\rightarrow\infty$). Let $P[f]$ be a differential polynomial of $f$ having differential order $k$ and with meromorphic coefficients $b_j$ as small functions of $f.$ Assume that $P[f]$ is non-constant. Then the following conclusions hold:
\begin{itemize}
\item[(i)] If $\nu_P>1,$ then either $f$ or $P[f]-\psi$ has infinitely many zeros in $\mathbb{C}.$ Moreover, if $f$ is entire, then we can take $\nu_P\geq 1.$
\item[(ii)] If $f$ has zeros of multiplicity at least $k+1,$ and $(k+2)\nu_P>\Gamma_P+2,$ then $P[f]-\psi$ has infinitely many zeros in $\mathbb{C}.$ If $f$ is entire, then the condition $(k+2)\nu_P>\Gamma_P+2$ can be replaced by the condition $(k+2)\nu_P>\Gamma_P+1.$ 
\end{itemize}
\end{lemma}

The conditions $\nu_p>1$ in (i) and $(k+2)\nu_P>\Gamma_P+2$ in (ii) are essential. For example, let $f(z)=e^z/(1+e^z).$ Then $f(z)\neq 0,~1.$ Let $P[f]:=2f-f^2$ so that $\nu_p=1$ and $(k+2)\nu_P<\Gamma_P+2.$ However, $P[f](z)=(2e^z+e^{2z})/(1+e^z)^2\neq 1.$

\begin{proof}[\bf Proof of Lemma \ref{lem:3}]
Write $P[f]$ in the form $$P[f]=H_{\nu_P}[f]+H_{\nu_P+1}[f]+\cdots+H_{\gamma_P}[f],$$ where $H_j[f]$'s are homogeneous differential polynomials of $f$ of degree $j.$ 

Since $f\not\equiv 0,$ we have $$m\left(r,\frac{P[f]}{f^{\gamma_P}}\right)=m\left(r,\frac{H_{\nu_P}[f]+H_{\nu_P+1}[f]+\cdots+H_{\gamma_P}[f]}{f^{\gamma_P}}\right).$$ From the Nevanlinna's theorem on logarithmic derivative, we get $$m\left(r,\frac{H_j[f]}{f^j}\right)=S(r,f)~\forall~j.$$ Then 
\begin{align}\label{eq:l3_1}
m\left(r,\frac{P[f]}{f^{\gamma_P}}\right)&\leq m\left(r,\frac{H_{\nu_P}[f]+H_{\nu_P+1}[f]+\cdots+H_{\gamma_P-1}[f]}{f^{\gamma_P-1}}\right)+ m\left(r,\frac{1}{f}\right) + S(r,f)\nonumber\\ 
&\leq 2~m\left(r,\frac{1}{f}\right) + m\left(r,\frac{H_{\nu_P}[f]+H_{\nu_P+1}[f]+\cdots+H_{\gamma_P-2}[f]}{f^{\gamma_P-2}}\right) + S(r,f)\nonumber\\
&\vdots\nonumber\\ 
&\leq (\gamma_P-\nu_P)m\left(r,\frac{1}{f}\right)+ S(r,f).
\end{align}
Now by the first fundamental theorem of Nevanlinna and \eqref{eq:l3_1}, we have
\begin{align*}
\gamma_P~T(r,f) &= \gamma_P~m\left(r,\frac{1}{f}\right)+\gamma_P~N\left(r,\frac{1}{f}\right)+O(1)\\
&\leq m\left(r,\frac{P[f]}{f^{\gamma_P}}\right)+ m\left(r,\frac{1}{P[f]}\right) +\gamma_P~N\left(r,\frac{1}{f}\right)+ \log{2}+O(1)\\
&\leq (\gamma_P-\nu_P)~m\left(r,\frac{1}{f}\right)+ m\left(r,\frac{1}{P[f]}\right) +\gamma_P~N\left(r,\frac{1}{f}\right)+ S(r,f)\\
&= (\gamma_P-\nu_P)~T(r,f)-(\gamma_P-\nu_P)~N\left(r,\frac{1}{f}\right) + T(r,P[f])\\ &\qquad
-N\left(r,\frac{1}{P[f]}\right)+ \gamma_P~N\left(r,\frac{1}{f}\right)+S(r,f).
\end{align*}
That is, $$\nu_P~T(r,f)\leq\nu_P~N\left(r,\frac{1}{f}\right)+ T(r,P[f])-N\left(r,\frac{1}{P[f]}\right)+S(r,f).$$
Applying the second fundamental theorem of Nevanlinna for small functions to $T(r,P[f]),$ we get
\begin{align}\label{eq:l3_2}
\nu_P~T(r,f) &\leq\nu_P~N\left(r,\frac{1}{f}\right)+ \overline{N}\left(r,\frac{1}{P[f]}\right)+\overline{N}(r,P[f])\nonumber\\ &\qquad+\overline{N}\left(r,\frac{1}{P[f]-\psi}\right)-N\left(r,\frac{1}{P[f]}\right)+S(r,f)\nonumber\\
&\leq \nu_P~N\left(r,\frac{1}{f}\right)+ \overline{N}(r,f)+ \overline{N}\left(r,\frac{1}{P[f]}\right)\nonumber\\ &\qquad+\overline{N}\left(r,\frac{1}{P[f]-\psi}\right)-N\left(r,\frac{1}{P[f]}\right)+S(r,f).
\end{align}
We now prove (i). 

First note that if $z_0$ is a zero of $f$ of multiplicity $m~(\geq 1),$ then $P[f]$ has a zero at $z_0$ of multiplicity at least $(m+1)\nu_P-\Gamma_P$ and hence $$N\left(r,\frac{1}{P[f]}\right)- \overline{N}\left(r,\frac{1}{P[f]}\right)\geq\left[(m+1)\nu_P-\Gamma_P-1\right]\overline{N}\left(r,\frac{1}{f}\right).$$
Then from \eqref{eq:l3_2}, we obtain
\begin{align*}
\nu_P~T(r,f) &\leq \nu_P~N\left(r,\frac{1}{f}\right)+\overline{N}(r,f)+\left[1+\Gamma_P-(m+1)\nu_P\right]\overline{N}\left(r,\frac{1}{f}\right)\\ &\qquad + \overline{N}\left(r,\frac{1}{P[f]-\psi}\right)+ S(r,f).\\
&= \overline{N}(r,f)+ \left(1+\Gamma_P-\nu_P\right)\overline{N}\left(r,\frac{1}{f}\right) + \overline{N}\left(r,\frac{1}{P[f]-\psi}\right)+ S(r,f).
\end{align*}
Thus $$(\nu_P-1)T(r,f)\leq\left(1+\Gamma_P-\nu_P\right)\overline{N}\left(r,\frac{1}{f}\right)+ \overline{N}\left(r,\frac{1}{P[f]-\psi}\right)+ S(r,f).$$
This proves (i).

To prove (ii), suppose that $f$ has zeros of multiplicity at least $k+1$ and $(k+2)\nu_P>\Gamma_P+2.$ Then one can easily see that
$$N\left(r,\frac{1}{P[f]}\right)- \overline{N}\left(r,\frac{1}{P[f]}\right)\geq\left[(k+2)\nu_P-\Gamma_P-1\right]\overline{N}\left(r,\frac{1}{f}\right).$$ Again, from \eqref{eq:l3_2}, we obtain
\begin{align*}
\nu_P~T(r,f) &\leq \nu_P~N\left(r,\frac{1}{f}\right)+\overline{N}(r,f)+\left[1+\Gamma_P-(k+2)\nu_P\right]\overline{N}\left(r,\frac{1}{f}\right)\\ &\qquad + \overline{N}\left(r,\frac{1}{P[f]-\psi}\right)+ S(r,f).\\
&\leq \nu_p~T(r,f)+ T(r,f)+\left[1+\Gamma_P-(k+2)\nu_P\right]T(r,f)\\ &\qquad + \overline{N}\left(r,\frac{1}{P[f]-\psi}\right)+ S(r,f).
\end{align*}
This implies that $$\left[(k+2)\nu_P-\Gamma_P-2\right]~T(r,f)\leq \overline{N}\left(r,\frac{1}{P[f]-\psi}\right)+ S(r,f)$$ and hence
$$T(r,f)\leq\frac{1}{(k+2)\nu_P-\Gamma_P-2}\overline{N}\left(r,\frac{1}{P[f]-\psi}\right)+S(r,f).$$
\end{proof}

\section{Proofs of Theorems}

\begin{proof}[\bf Proof of Theorem \ref{thm:2}]
Suppose that $\mathcal{F}$ is not normal at $z_0\in D.$ Then by Lemma \ref{lem:zp}, there exist a sequence $\left\{f_n\right\}\subset\mathcal{F},$ a sequence of points $\left\{z_n\right\}\subset D$ with $z_n\rightarrow z_0$ and a sequence of positive real numbers satisfying $\rho_n\rightarrow 0$ such that the sequence $$\phi_n(\zeta):=\frac{f_n(z_n+\rho_n\zeta)}{\rho_n^{\alpha}}\rightarrow \phi(\zeta),~\alpha:=\Theta_P-1,$$ locally uniformly in $\mathbb{C},$ where $\phi$ is a non-constant entire function of  finite order with zeros of multiplicity at least $k+1.$ Moreover, the corresponding sequence $\left\{g_n\right\}\subset\mathcal{G}$ converges locally uniformly to $g$ in $D,$ where $g~(\not\equiv\infty)$ has zeros of multiplicity at least $k+1.$ 

Note that 
\begin{align*}
\tilde{P}[\phi_n](\zeta) &:= P[f_n](z_n+\rho_n\zeta)\\
&= \sum\limits_{i=1}^{m}a_i(z_n+\rho_n\zeta)~\rho_n^{\left[(1+\alpha)\gamma_{M_i}-\Gamma_{M_i}\right]}~M_i[\phi_n](\zeta).
\end{align*}
Since $$\alpha=\Theta_P-1=\frac{\Gamma_{M_1}}{\gamma_{M_1}}-1~\mbox{ and }~ \frac{\Gamma_{M_1}}{\gamma_{M_1}}>\frac{\Gamma_{M_t}}{\gamma_{M_t}}~ \mbox{ for }~ 2\leq t\leq m,$$ we have
\begin{align*}
\tilde{P}[\phi_n](\zeta) &= P[f_n](z_n+\rho_n\zeta)\\
&= a_1(z_n+\rho_n\zeta)M_1[\phi_n](\zeta)+ \sum\limits_{i=2}^{m}a_i(z_n+\rho_n\zeta)~\rho_n^{\left[(1+\alpha)\gamma_{M_i}-\Gamma_{M_i}\right]}~M_i[\phi_n](\zeta).
\end{align*}
Again, since all $a_i~(1\leq i\leq m)$ are holomorphic functions in $D,$ it follows that $$\sum\limits_{i=2}^{m}a_i(z_n+\rho_n\zeta)~\rho_n^{\left[(1+\alpha)\gamma_{M_i}-\Gamma_{M_i}\right]}~M_i[\phi_n](\zeta)$$ converges locally uniformly to $0$ in $\mathbb{C}$ and hence $$\tilde{P}[\phi_n](\zeta)=P[f_n](z_n+\rho_n\zeta)\rightarrow a_1(z_0)M_1[\phi](\zeta)$$ locally uniformly in $\mathbb{C}.$ Since $a_1(z_0)\neq 0,$ we may assume, without loss of generality, that $a_1(z_0)=1.$ Also, since $g_n\rightarrow g$ implies that $P[g_n]\rightarrow P[g]$ locally uniformly in $\mathbb{C}.$ Since $\phi$ is an entire function having zeros of multiplicity at least $k+1,$ $M_1[\phi]$ assumes every non-zero value in $\mathbb{C}.$ In particular, $M_1[\phi]$ assumes the non-zero value $a,$ say. Let $\zeta_0\in\mathbb{C}$ be such that $M_1[\phi](\zeta_0)=a.$ Since $M_1[\phi]\not\equiv a,$ by Hurwitz's theorem, there exists a sequence $\zeta_{n}\rightarrow\zeta_0$ such that for sufficiently large $n,$ $$P[f_n](z_n+\rho_n\zeta_n)=a$$ and so by hypothesis, we have $$P[g_n](z_n+\rho_n\zeta_n)=a.$$ Taking $n\rightarrow\infty,$ we get $P[g](\zeta_0)=a\neq 0.$ 

Next, we claim that $\phi$ omits $0.$ For, let $\zeta_0\in\mathbb{C}$ be such that $\phi(\zeta_0)=0.$ Since $\phi\not\equiv 0,$ by Hurwitz's theorem, there exists $\tilde{\zeta_n}\rightarrow\zeta_0$ such that for sufficiently large $n,$ $\phi_n(\tilde{\zeta_n})=\rho_n^{-\alpha}f_n(z_n+\rho_n\tilde{\zeta_n})=0.$ This implies that $f_n(z_n+\rho_n\tilde{\zeta_n})=0$ and so by hypothesis, $g_n(z_n+\rho_n\tilde{\zeta_n})=0.$ Taking $n\rightarrow\infty,$ we get $g(\zeta_0)=0.$ Since $g$ has zeros of multiplicity at least $k+1,$ it follows that $P[g](z_0)=0,$ a contradiction. This establishes the claim.

Since $\phi(\zeta)\neq 0,$ $\phi$ is transcendental and so from Lemma \ref{lem:3} (i), it follows that $M_1[\phi]$ assumes $a$ infinitely many times. If $M_1[\phi](\zeta_0)=a,$ then as done before, we would get $P[g](\zeta_0)=a.$ Let $l~(\geq 1)$ be the multiplicity of zero of $P[g]-a$ at $\zeta_0.$ Since $P[g]\not\equiv a,$ by Hurwitz's theorem, for sufficiently large $n,$ $P[g_n]-a$ has exactly $l$ zeros in a neighbourhood $U(\zeta_0)$ of $\zeta_0.$ Since $M_1[\phi]$ assumes $a$ infinitely many times, we assume that $\zeta_1,\ldots,\zeta_{l+1}$ be the distinct zeros of $M_1[\phi]-a.$ Again, by Hurwitz's theorem, for sufficiently large $n,$ there exists $\zeta_{j,n}\rightarrow\zeta_j$ such that $P[\phi_n](\zeta_{j,n})=a$ for $j=1,\ldots,l+1.$ Thus there are $l+1$ distinct zeros of $P[\phi_n]-a$ in $U(\zeta_0)$ and hence $P[g_n]-a$ has $l+1$ distinct zeros in $U(\zeta_0),$ a contradiction.
\end{proof}

\begin{proof}[\bf Proof of Theorem \ref{thm:3}]
Assume that $\mathcal{F}$ is not normal at $z_0\in D.$ Then by Lemma \ref{lem:zp}, there exist a sequence $\left\{f_n\right\}\subset\mathcal{F},$ a sequence of points $\left\{z_n\right\}\subset D$ with $z_n\rightarrow z_0$ and a sequence of positive real numbers satisfying $\rho_n\rightarrow 0$ such that the sequence $$\phi_n(\zeta):=\frac{f_n(z_n+\rho_n\zeta)}{\rho_n^{\alpha}}\rightarrow \phi(\zeta),~\alpha:=\Theta_P-1,$$ locally uniformly in $\mathbb{C},$ where $\phi$ is a non-constant meromorphic function of finite order with zeros of multiplicity at least $k+1.$ Moreover, the corresponding sequence $\left\{g_n\right\}\subset\mathcal{G}$ converges locally uniformly to $g$ in $D,$ where $g~(\not\equiv\infty)$ has zeros of multiplicity at least $k+1.$ Then by a similar argument as in the proof of Theorem \ref{thm:2}, we find that on every compact subset of $\mathbb{C}$ not containing poles of $\phi,$ $$\tilde{P}[\phi_n](\zeta):=P[f_n](z_n+\rho_n\zeta)\rightarrow a_1(z_0)M_1[\phi](\zeta)~\mbox{as } n\rightarrow\infty.$$ Also, for every $\zeta\in\mathbb{C}\setminus\left\{g^{-1}(\infty)\right\},$ we have $P[g_n](\zeta)\rightarrow P[g](\zeta)~\mbox{as } n\rightarrow\infty$ locally uniformly. Since $M_1[\phi]=\prod\limits_{j=0}^{k}\left(\phi^{(j)}\right)^{n_{j1}},$ we shall denote, for the sake of convenience, $M_1[\phi]$ by $M[\phi]$ and the powers $n_{j1}$ by $n_j$ for $j=1,\ldots,k.$ Also, we may assume that $a_1(z_0)=1.$ Now, we consider the following cases:\\

{\bf Case 1:} $n_0=n_1=\cdots=n_{k-1}=0.$ In this case, we have $M[\phi]=\left(\phi^{(k)}\right)^{n_k},$ where $n_k\geq 1.$\\
 {\it Claim 1:} $\phi(\zeta)\neq 0.$

For, let $\zeta_0\in\mathbb{C}$ be such that $\phi(\zeta_0)=0.$ Since $\phi\not\equiv 0,$ by Hurwitz's theorem, there exists $\zeta_n\rightarrow\zeta_0$ such that for sufficiently large $n,$ $\phi_n(\zeta_n)=\rho_n^{-\alpha}f_n(z_n+\rho_n\zeta_n)=0.$ This implies that $f_n(z_n+\rho_n\zeta_n)=0$ and so by hypothesis, $g_n(z_n+\rho_n\zeta_n)=0.$ Taking $n\rightarrow\infty,$ we get $g(z_0)=0.$ Since $g$ has zeros of multiplicity at least $k+1,$ it follows that $P[g](z_0)=0.$ If $\phi$ is transcendental, then by Lemma \ref{lem:w1}, $M[\phi]$ assumes the non-zero complex number $a.$ Let $\zeta'\in\mathbb{C}$ such that $M[\phi](\zeta')=a.$ Since $M[\phi]\not\equiv a,$ by Hurwitz's theorem, there exists a sequence $\zeta_{n}'\rightarrow\zeta'$ such that for sufficiently large $n,$ $P[f_n](z_n+\rho_n\zeta_{n}')=a$ and hence by hypothesis, we get $P[g_n](z_n+\rho_n\zeta_{n}')=a.$ Taking $n\rightarrow\infty,$ we get $P[g](z_0)=a\neq 0,$ a contradiction. Hence $\phi$ must be rational. Furthermore, using the preceding argument, one can easily see that $\phi$ is not a polynomial and therefore there exists $\zeta^*\in\mathbb{C}$ such that $\phi(\zeta^*)=\infty.$ Again, by Hurwitz's theorem, there exists a sequence $\zeta_{n}^*\rightarrow\zeta^*$ such that for sufficiently large $n,$ we have $f_n(z_n+\rho_n\zeta_n^*)=\infty$ and hence $g_n(z_n+\rho_n\zeta_n^*)=\infty.$ Taking $n\rightarrow\infty,$ we get $g(z_0)=\infty,$ which is again a contradiction. Hence $\phi(\zeta)\neq 0.$\\
{\it Claim 2:} $\phi$ has no poles.

Suppose there is some $\zeta_1\in\mathbb{C}$ such that $\phi(\zeta_1)=\infty.$ Then as done before, we get $g(z_0)=\infty.$ Since $\phi(\zeta)\neq 0,$ by Lemma \ref{lem:h}, there is some $\zeta_2\in\mathbb{C}$ such that $M[\phi](\zeta_2)=a.$ Again, by Hurwitz's theorem, there exists a sequence $\zeta_{n,2}\rightarrow\zeta_2$ such that for sufficiently large $n,$ we have $P[f_n](z_n+\rho_n\zeta_{n,2})=a.$ This implies that $P[g_n](z_n+\rho_n\zeta_{n,2})=a$ and hence $P[g](z_0)=a$ which is not possible since $P[g]\not\equiv a$ and $g(z_0)=\infty.$ This establishes Claim 2. 

Now, Claim 1 and Claim 2 together imply that $\phi$ is a non-vanishing transcendental entire function and so from Lemma \ref{lem:h}, it follows that $M[\phi]$ assumes $a$ infinitely many times. Now, proceeding similarly as in the proof of Theorem \ref{thm:2}, we get a contradiction.\\

{\bf Case 2:} $n_0\geq 1,~n_k\geq 1.$ Then $\gamma_M\geq 2$ and a simple computation shows that
\begin{align*}
(k+2)\gamma_M-\Gamma_M-2 &= (k+2)(n_0+n_1+\cdots+n_k)-[n_0+2n_1+\cdots+(k+1)n_k]-2\\
&= (k+2-1)n_0 + (k+2-2)n_1+\cdots+ [k+2-(k+1)]n_k-2\\
&\geq kn_0+ (k+2-2)n_1 + \cdots+ [(k+2)-(k-1)]n_{k-1}\\
&\geq 1. 
\end{align*} 
Now using the same arguments as that of Case 1 and Lemma \ref{lem:3} (ii), we obtain $\phi(\zeta)\neq 0$ and hence $\phi$ is not a polynomial. If $\phi$ is transcendental or rational, then we use Lemma \ref{lem:3} (i), and the arguments from Claim 2 to conclude that $\phi$ is a non-vanishing transcendental entire function. Thus from Lemma \ref{lem:3} (i), it follows that $M[\phi]$ assumes $a$ infinitely many times. Again, proceeding similarly as in the proof of Theorem \ref{thm:2}, we get a contradiction. Hence $\mathcal{F}$ is normal in $D.$
\end{proof}

\begin{proof}[\bf Proof of Theorem \ref{thm:1}]
Suppose that $\mathcal{F}$ is not normal at $z_0\in D.$ Then we consider the following two cases:\\

{\bf Case 1:} $h(z_0)\neq 0.$ By Lemma \ref{lem:zp}, there exist a sequence $\left\{f_j\right\}\subset\mathcal{F},$ a sequence of points $\left\{z_j\right\}\subset D$ with $z_j\rightarrow z_0$ and a sequence of positive real numbers satisfying $\rho_j\rightarrow 0$ such that the sequence $$g_j(\xi):=\frac{f_j(z_j+\rho_j\xi)}{\rho_j^{\alpha}}\rightarrow g(\xi),~\alpha:=\frac{\Gamma_M}{\gamma_M}-1,$$ locally uniformly in $\mathbb{C}$ with respect to the spherical metric, where $g\in\mathcal{M}(\mathbb{C})$ is a non-constant function of order at most $2.$ Since each $f_j$ has zeros of multiplicity at least $k+m+1,$ by Argument principle, it follows that $g$ has zeros of multiplicity at least $k+m+1.$ Furthermore, a simple computation shows that
\begin{align*}
M[g_j](\xi)&=\prod\limits_{i=0}^{k}\left(\left(\frac{f_j(z_j+\rho_j\xi}{\rho_j^{\alpha}}\right)^{(i)}\right)^{n_i}\\
&= \frac{\rho_j^{n_1+2n_2+\cdots+kn_k}}{\rho_j^{\alpha(n_0+n_1+\cdots+n_k)}}\prod\limits_{i=0}^{k}\left(\left(f_j(z_j+\rho_j\xi)\right)^{(i)}\right)^{n_i}\\
&= \frac{\rho_j^{\Gamma_M-\gamma_M}}{\rho_j^{\alpha\gamma_M}}M[f_j](z_j+\rho_j\xi)\\
&= M[f_j](z_j+\rho_j\xi).
\end{align*}
Hence on every compact subset of $\mathbb{C}$ not containing poles of $g,$ we have 
$$M[g_j](\xi)=M[f_j](z_j+\rho_j\xi)\rightarrow M[g](\xi).$$ 
Obviously, $M[g](\xi)\not\equiv h(z_0).$ Also, one can easily check that $(k+2)\gamma_M-\Gamma_M-2>0.$ Thus by Lemma \ref{singh} and Lemma \ref{lem:3} (ii), $M[g](\xi)-h(z_0)$ has at least two distinct zeros, say $\xi_1$ and $\xi_2.$ Choose $\delta>0$ small enough such that $D(\xi_1,\delta) \cap D(\xi_2,\delta)=\phi$ and for any $\xi\neq \xi_1,~\xi_2$ in $D(\xi_1,\delta)\cup D(\xi_2,\delta),$ $M[g](\xi)-h(z_0)\neq 0.$ By Hurwitz's theorem, there exist sequences $\xi_{j,1}\rightarrow\xi_1$ and $\xi_{j,2}\rightarrow\xi_2$ such that for sufficiently large $j,$ $$M[f_j](z_j+\rho_j\xi_{j,l})-h(z_j+\rho_j\xi_{l})=0,~l=1,~2.$$ By hypothesis, it follows that for any fixed integer $p$ and for all $j,$ $$M[f_p](z_j+\rho_j\xi_{j,l})-h(z_j+\rho_j\xi_{l})=0,~l=1,~2.$$ Taking $j\rightarrow\infty,$ we get $$M[f_p](z_0)-h(z_0)=0.$$ Since the zeros of $M[f_p]-h(z_0)$ have no accumulation point, for sufficiently large $j,$ we obtain $z_j+\rho_j\xi_{j,1}=0=z_j+\rho_j\xi_{j,2}.$ This implies that $\xi_{j,1}=-z_j/\rho_j=\xi_{j,2}$ showing that $D(\xi_1,\delta)\cap D(\xi_2,\delta)\neq\phi,$ a contradiction.\\

{\bf Case 2:} $h(z_0)=0.$ Then $h(z)=z^{t}h_1(z),$ where $h_1$ is holomorphic in $D$ such that $h_1(z_0)\neq 0$ and $1\leq t\leq m.$ We may assume that $h_1(z_0)=1.$ Again, by Lemma \ref{lem:zp}, there exist a sequence $\left\{f_j\right\}\subset\mathcal{F},$ a sequence of points $\left\{z_j\right\}\subset D$ with $z_j\rightarrow z_0$ and a sequence of positive real numbers satisfying $\rho_j\rightarrow 0$ such that the sequence $$g_j(\xi):=\frac{f_j(z_j+\rho_j\xi)}{\rho_j^{\beta}}\rightarrow g(\xi),~\beta:=\frac{\Gamma_M+t}{\gamma_M}-1,$$ locally uniformly in $\mathbb{C}$ with respect to the spherical metric, where $g\in\mathcal{M}(\mathbb{C})$ is non-constant. We consider the following subcases:

{\bf Subcase 2.1:} Suppose that there exists a subsequence of $z_j/\rho_j$ (which, for the sake of convenience, is again denoted by $z_j/\rho_j$) such that $z_j/\rho_j\rightarrow\infty$ as $j\rightarrow\infty.$

 Consider the family $$V_j(\xi):=\frac{f_j(z_j+z_j\xi)}{z_j^{\beta}}.$$ By simple calculation, one can see that $M[f_j](z_j+z_j\xi)=z_j^{t}M[V_j](\xi).$ By hypothesis, for any distinct integers $r,~s,$ we have, $$M[f_r](z)-h(z)=0\Leftrightarrow~M[f_s](z)-h(z)=0.$$ Choose $\xi,$ say $\xi_i$ such that for some $j,$ say $j_i,$ $z_{j_i}+z_{j_i}\xi_{i}$ is a zero of $M[f_r](z)-h(z)$ and hence of $M[f_s](z)-h(z).$ Then 
\begin{align*}
& M[f_r](z_{j_i}+z_{j_i}\xi_i)=h(z_{j_i}+z_{j_i}\xi_i)\Leftrightarrow~M[f_s](z_{j_i}+z_{j_i}\xi_i)=h(z_{j_i}+z_{j_i}\xi_i)\\
&\Rightarrow M[V_r](\xi_i)=(1+\xi_i)^t h_1(z_{j_i}+z_{j_i}\xi_i)\Leftrightarrow~M[V_s](\xi_i)=(1+\xi_i)^t h_1(z_{j_i}+z_{j_i}\xi_i).
\end{align*}
Since $(1+\xi_i)^th_1(z_{j_i}+z_{j_i}\xi_i)\rightarrow (1+\xi_i)^t\neq 0$ in a sufficiently small neighbourhood $N_0$ of the origin, it follows that $(1+\xi_i)^t h_1(z_{j_i}+z_{j_i}\xi_i)\neq 0$ in $N_0.$ Thus by Case 1, $\left\{V_j\right\}$ is normal in $N_0$ and so there exists a subsequence of $\left\{V_j\right\}$ (again denoted by $\left\{V_j\right\}$) such that $V_j$ converges spherically locally uniformly in $N_0$ to a meromorphic function $V.$\\

{\bf Subcase 2.1.1:} $V(0)\neq 0.$ Then for any $\xi\in\mathbb{C},$ $$g_j(\xi)=\rho_{j}^{-\beta}f_j(z_j+\rho_j\xi)=\left(\frac{z_j}{\rho_j}\right)^{\beta}V_j\left(\frac{\rho_j}{z_j}\xi\right)$$ converges spherically locally uniformly to $\infty,$ showing that $g\equiv\infty,$ a contradiction.\\

{\bf Subcase 2.1.2:} $V(0)=0.$ Then $V'(0)\neq\infty$ and so $$g_j'(\xi)=\rho_{j}^{-\beta+1}f_j'(z_j+\rho_j\xi)=\left(\frac{\rho_j}{z_j}\right)^{-\beta+1}V_j'\left(\frac{\rho_j}{z_j}\xi\right)$$ converges spherically locally uniformly in $\mathbb{C}$ to $0.$ This shows that $g'\equiv 0,$ a contradiction to the fact that $g$ is non-constant.\\

{\bf Subcase 2.2:} Suppose there exists a subsequence of $z_j/\rho_j$ (which, for the sake of convenience, is again denoted by $z_j/\rho_j$) such that $z_j/\rho_j\rightarrow\zeta$ as $j\rightarrow\infty,$ where $\zeta\in\mathbb{C}.$

 Consider the family $W_j(\xi):=\rho_j^{-\beta}f_j(\rho_j\xi).$ Then $$W_j(\xi)=g_j\left(\xi-\frac{z_j}{\rho_j}\right)\rightarrow g(\xi-\zeta):=W(\xi).$$ It easily follows that on any compact subset of $\mathbb{C},$ containing no poles of $W,$ $$M[W_j](\xi)-\xi^t\rightarrow M[W](\xi)-\xi^t.$$ 
 Clearly, $M[W](\xi)\not\equiv\xi^t.$ Thus by Lemma \ref{lem:2}, $M[W](\xi)-\xi^t$ has at least two distinct zeros. Proceeding similarly as in Case 1, we get a contradiction. Hence $\mathcal{F}$ is normal in $D.$
\end{proof}

\section{Statements and Declarations}
\begin{itemize}
\item[]{\bf Funding:} The authors declare that no funds, grants, or other support were received during the preparation of this manuscript.

\item[]{\bf Competing Interests:} The authors declare that they have no conflict of interests. The authors have no relevant financial or non-financial interests to disclose.

\item[]{\bf Data availability:} Data sharing is not applicable to this article as no datasets were generated or analysed during the current study.
\end{itemize}

\end{document}